\documentclass[10pt]{article}
\usepackage{mathrsfs}
\usepackage{amsfonts}
\usepackage{amsthm,amsmath,amssymb,anysize}
\newtheorem{lemma}{Lemma}[section]
\newtheorem{theorem}[lemma]{Theorem}

\newtheorem{proposition}[lemma]{Proposition}

\usepackage[numbers,sort&compress]{natbib}
\marginsize{4.2cm}{4.2cm}{3.6cm}{3cm}
\numberwithin{equation}{section}

\title{\textsf{Automorphism
groups of the finite dimensional special odd Hamiltonian
superalgebras in prime characteristic}}
\author{\textsc{Liming Tang$^{1,2}$ \textsc{and}
    \textsc{Wende Liu$^{1,2,}$}}\footnote{Correspondence:  wendeliu@ustc.edu.cn (W. Liu), limingtang@hrbnu.edu.cn (L. Tang)
 } \footnote{Supported by  the NSF
  of China (10871057) and the NSF of Heilongjiang Province, China (A200802)}\;\; \;
  \\
  \\
  \textit{$^{1}$Department of Mathematics},
  \textit{Harbin Institute of Technology}\\
  \textit{Harbin 150006, China}\\
\\
  \ \ \textit{$^{2}$School of Mathematical Sciences},
  \textit{Harbin Normal University} \\
  \textit{Harbin 150025, China}}
\date{ }
\begin{document}
\maketitle
\begin{quotation}
\small\noindent \textbf{Abstract}: This paper is devoted to a study
of the automorphism groups of three series of finite dimensional
special odd Hamiltonian superalgebras $\mathfrak{g}$ over a field of
prime characteristic. Our aim is to characterize the connections
between the automorphism groups of $\mathfrak{g}$ and the
automorphism groups of the corresponding underlying superalgebras.
Precisely speaking, we embed the former into the later. Moreover, we
determine the images of the normal series of the automorphism groups
and homogeneous automorphism groups of $\mathfrak{g.}$

\vspace{0.2cm} \noindent{\textbf{Keywords}}: Special odd Hamiltonian
superalgebras; Automorphism groups

\vspace{0.1cm} \noindent \textbf{Mathematics Subject Classification
2000}: 17B50, 17B40

\end{quotation}
 \setcounter{section}{-1}
\section{Introduction}

As is well known, the classification problem is still open for
finite dimensional simple module Lie superalgebras (see \cite{bl})
and the simple graded Lie superalgebras of Cartan type are quite
useful in attempts to classify finite dimensional simple Lie
superalgebras over a field of prime characteristic (see
\cite{FJZ,lh,lz,lz2}, for example). The automorphism groups of
Cartan type Lie superalgebras play a vital role in future studies on
structures and representations of Lie superalgebras of Cartan type
(see \cite{lz, lz2}, for example). In \cite{lz}, the automorphism
groups were determined for finite dimensional restricted modular Lie
superalgebras $W,$ $ S,$ $ H,$ $ K.$ In \cite{lz2}, the automorphism
groups of finite dimensional restricted odd Hamiltonian
superalgebras $HO$ were given.

  The purpose of this paper is to discuss the automorphism groups
  of restricted special odd Hamiltonian superalgebras in prime
  characteristic. Our methods are modeled on those used by \cite{lz}.
   There is a little difference from those previous papers:
   We use more general methods to deal with three series of finite dimensional
special odd Hamiltonian superalgebras $\mathfrak{g}$ containing
non-simple cases. In additional, one should stress that some
conclusions are general which provide information for investigating
other problems of Lie superalgebras.

\section{Basics}
Let us first introduce the notation which will be used throughout
this paper. Quite generally, we adopt the conventions of \cite{lz}.
In the following, $\mathbb{F}$ is a field of characteristic $p>3.$
Let $\mathcal{O}(m;\underline{t})$ be the \textit{divided power
algebra} over $\mathbb{F}$ with basis $\{x^{(\alpha)}\mid\alpha\in
\mathbb{A}\}\,\ \mbox{where}\,\
\mathbb{A}:=\{\alpha\in\mathbb{N}^{m}\mid\alpha_{i}\leq\pi_{i}\}.$
Let $\Lambda(m)$ be the \textit{exterior superalgebra} over
$\mathbb{F}$ of $m$ variables $ x_{m+1},x_{m+2},\ldots,x_{2m}.$ Let
$$\mathbb{B}:=\{\langle i_{1},i_{2},\ldots,i_{k}\rangle \mid m+1\leq
i_{1}<i_{2}<\cdots<i_{k}\leq 2m, 1\leq k\leq m\}.$$  For $u:=\langle
i_{1},i_{2},\ldots,i_{k}\rangle \in\mathbb{B},$ set $|u|:=k$ and
write $x^{u}:=x_{i_{1}}x_{i_{2}} \cdots x_{i_{k}}.$ The tensor
product $\mathcal{O}(m,m;\underline{t}):=\mathcal
{O}(m;\underline{t})\otimes\Lambda(m)$
 is an associative superalgebra. Note that $\mathcal{O}(m,m;\underline{t})$ has a
standard $\mathbb{F}$-basis $\{x^{(\alpha)}x^{u}\mid
(\alpha,u)\in\mathbb{A}\times\mathbb{B}\}.$ For convenience, put
$$\mathbf{I}_{0}:={1,2,\ldots,m};\,\ \mathbf{I}_{1}:={m+1,\ldots,2m},\,\ \mbox{and}\,\
\mathbf{I}=\mathbf{I}_{0}\cup \mathbf{I}_{1}.$$ Let $D_{r}$ be the
superderivation of $\mathcal{O}(m,m;\underline{t})$ such that
$$D_{r}(x^{(\alpha)})=x^{(\alpha-\varepsilon_{r})}\,\ \mbox{for}\,\ r\in
\mathbf{I}_{0}\,\ \mbox{and}\,\ D_{r}(x_{s})=\delta_{rs}\,\
\mbox{for}\,\ r,s\in \mathbf{I}.$$ The generalized Witt superalgebra
$W(m,m;\underline{t})$ is a free
$\mathcal{O}(m,m;\underline{t})$-module with basis $\{D_{r}\mid r\in
\mathbf{I}\}.$  In particular, $W(m,m;\underline{t})$ has a standard
$\mathbb{F}$-basis $\{x^{(\alpha)}x^{u}D_{r}\mid (\alpha,u,r)\in
\mathbb{A}\times \mathbb{B}\times \mathbf{I}\}.$ Note that
$\mathcal{O}(m,m;\underline{t})$ possesses a standard
$\mathbb{Z}$-grading structure
$\mathcal{O}(m,m;\underline{t}):=\bigoplus_{r=0}^{\xi}\mathcal{O}(m,m;\underline{t})_{[r]}$
by letting
$$\mathcal{O}(m,m;\underline{t})_{[r]}:={\rm{span}}_{\mathbb{F}}\{x^{(\alpha)}x^{u}\mid
|\alpha|+|u|=r\},\,\ \xi£»=|\pi|+m=\sum_{i\in
\mathbf{Y}_{0}}p^{t_{i}}.$$ This induces naturally a
$\mathbb{Z}$-grading structure of
$W(m,m;\underline{t})=\bigoplus_{i={-1}}^{\xi-1}W(m,m;\underline{t})_{[i]}$
where
$$W(m,m;\underline{t})_{[i]}:={\rm{span}}_{\mathbb{F}}\{fD_{r}\mid f\in \mathcal{O}(m,m;\underline{t})_{[i+1]}, r\in
\mathbf{I}\}.$$We usually simplify the notion and write
$\mathcal{O}$ and $W$
     instead of
    $\mathcal{O}(m,m;\underline{t})$ and $W(m,m;\underline{t}),$
    respectively.
 For a vector superspace $V=V_{\overline{0}}\bigoplus
V_{\overline{1}},$ we write ${\rm{p}}(x):=\theta$ for the parity of
a homogeneous element $x\in V_{\theta},\,\ \theta\in
\mathbb{Z}_{2}.$ We assume throughout that the symbol
$\mathrm{p}(x)$ implies  $x$ is $\mathbb{Z}_{2}$-homogeneous. Now
let $\mathscr{A}=\bigoplus_{i\in \mathbb{Z}}\mathscr{A}_{[i]}$ be a
$\mathbb{Z}$-graded superalgebra. Write
$\mathscr{A}_{m}=\bigoplus_{i\geq{m}}\mathscr{A}_{[i]}$ for all
$i\in \mathbb{Z}.$ In the remainder of this paper, the filtration
$(\mathscr{A}_{m})_{m\in \mathbb{Z}}$ associated with the given
$\mathbb{Z}$-grading of $\mathscr{A}$ is called the standard
filtration of $\mathscr{A}.$ If $L$ is a $\mathbb{Z}$-graded Lie
superalgebra, then the grading $(L_{[i]})_{i\in \mathbb{Z}}$ is
called $\textit{transitive}$ if $\{A\in L_{[i]}\mid
[A,L_{[-1]}]={0}\}=\{0\}$ for all $i\geq 0.$

      Take the involution of $\mathbf{I}$ satisfying that $i'=i+m$ for
      $i\in \mathbf{I}_{0}$ and $i'=i-m$ for $i\in \mathbf{I}_{1}.$ Define a linear operator
      ${\rm{T_{H}}}:\mathcal{O}\longrightarrow W$
      such that
      $${\rm{T_{H}}}(a):=\sum_{i\in
      \mathbf{I}}(-1)^{{\rm{p}}(D_{i}){\rm{p}}(a)}D_{i}(a)D_{i^{'}}
    \,\ \mbox{for}\,\ a\in\mathcal{O}.$$
   Note that ${\rm{T_{H}}}$ is odd and that
      $$\big[{\rm{T_{H}}}(a),{\rm{T_{H}}}(b)\big]={\rm{T_{H}}}{\rm{\big(T_{H}}}(a)(b)\big)
    \,\ \mbox{for
      all}
      \,\ a,b\in\mathcal{O}.$$
      Then $$HO(m,m;\underline{t}):=\{{\rm{T_{H}}}(a)\mid
      a\in\mathcal{O}\}$$ is a finite dimensional
      simple Lie superalgebra, which is called the\textit{ odd Hamiltonian
      superalgebra} (see \cite{lz2}). Let $${\rm{div}}:W\longrightarrow
      \mathcal{O}$$ be the divergence, which is a linear mapping, such that
      $${\rm{div}}(f_{r}D_{r})=(-1)^{{\rm{p}}(D_{r}){\rm{p}}(f_{r})}D_{r}(f_{r}).$$
      Note that ${\rm{div}}$ is an even superderivation of
      $W$ into the module
      $\mathcal{O}.$ Put $$
      S'(m,m;\underline{t}):=\{D\in W\mid
      {\rm{div}}(D)=0\}.$$
      Then $S'(m,m;\underline{t})$ is a $\mathbb{Z}$-graded subalgebra of
      $W.$ Its derived algebra is a simple Lie
      superalgebra, which is called the \textit{special
      superalgebra}. Let
      \begin{eqnarray*}
      SHO'(m,m;\underline{t})&:=&S'(m,m;\underline{t})\cap
      HO(m,m;\underline{t}),\\
      \overline{SHO}(m,m;\underline{t})&:=&[SHO'(m,m;\underline{t}),
      SHO'(m,m;\underline{t})],\\
      SHO(m,m;\underline{t})&:=&[\overline{SHO}(m,m;\underline{t}),
      \overline{SHO}(m,m;\underline{t})].
      \end{eqnarray*}
      We call these algebras the \textit{special odd Hamiltonian
      superalgebras} (see \cite{lh}). In the below, $\mathfrak{g}(m,m;\underline{t})$ will always be abbreviated to $\mathfrak{g},$ where $\mathfrak{g}= SHO',$
$\overline{SHO}$ or $SHO,$ unless other stated.
      A straightforward verification shows that $\mathfrak{g}$ is a $\mathbb{Z}$-graded subalgebra of $W,$ and \begin{eqnarray*}\mathfrak{g}&=&\bigoplus_{i\geq-1}\mathfrak{g}_{[i]},
      \,\ \mbox{where}\,\ \mathfrak{g}_{[i]}=\mathfrak{g}\cap W_{[i]}.
      \end{eqnarray*}

\section{Automorphism groups}\label{section2}
  We begin by introducing the necessary definitions concerning
  automorphism groups.

$\bullet$ \textit{Admissible automorphism groups}

Let $\mathscr{A}$ be a finite dimensional superalgebra over
$\mathbb{F}$ and $\mathcal{Q}$ a sub Lie superalgebra of the full
superderivation superalgebra ${\rm{Der}}\mathscr{A}.$ Put
$${\rm{Aut}}(\mathscr{A}:\mathcal{Q}):=\{\sigma\in {\rm{Aut}}\mathscr{A}\mid \widetilde{\sigma}(\mathcal{Q})\subset \mathcal{Q}\},$$
where $\widetilde{\sigma}(D):=\sigma D\sigma^{-1}$ for $D\in
\mathcal{Q}.$ Then ${\rm{Aut}}(\mathscr{A}:\mathcal{Q})$ is a
subgroup of ${\rm{Aut}}\mathscr{A},$ which is referred to as the
admissible automorphism group of $\mathscr{A}$ (to $\mathcal{Q}$).

$\bullet$ \textit{Homogeneous automorphism groups}

Denote by $\mathscr{A}$ a $\mathbb{Z}$-graded superalgebra. A
$\mathbb{Z}$-grading of $\mathscr{A}$ is a family of sub superspaces
$(\mathscr{A}_{[j]})_{j\in \mathbb{Z}}$ satisfying
$\mathscr{A}=\bigoplus_{i\in \mathbb{Z}}\mathscr{A}_{[i]}$ and
$\mathscr{A}_{[i]}\mathscr{A}_{[j]}\subset \mathscr{A}_{[i+j]}$. Set
$$
{\rm{Aut}}^{*}\mathscr{A}:=\{\sigma \in {\rm{Aut}}\mathscr{A} \mid
\sigma(\mathscr{A}_{[j]})\subset \mathscr{A}_{[j]},\forall j\in
\mathbb{Z}\}.$$ The subgroup ${\rm{Aut}}^{*}\mathscr{A}$ of
${\rm{Aut}}\mathscr{A}$ is called the homogeneous automorphism group
of $\mathscr{A}.$

$\bullet$ \textit{Standard normal series of ${\rm{Aut}}\mathscr{A}$}

Let $\mathscr{A}$ be a $\mathbb{Z}$-graded superalgebra. If the
standard filtration of $\mathscr{A}$ is invariant under automorphism
of $\mathscr{A}$, we have occasion to construct the standard normal
series of ${\rm{Aut}}\mathscr{A}.$ Put
 $${\rm{Aut}}_{i}\mathscr{A}:=\{\sigma\in {\rm{Aut}}\mathscr{A}\mid (\sigma-1)(\mathscr{A}_{j})\subset
\mathscr{A}_{i+j}, \forall j\in \mathbb{Z}\},\,\ i\geq0.$$ Then
${\rm{Aut}}_{i}\mathscr{A}$ is a normal subgroup of
${\rm{Aut}}\mathscr{A}$ for each $i\geq0$. Moreover, we call
${\rm{Aut}}_{0}\mathscr{A}> {\rm{Aut}}_{1}\mathscr{A}> \cdots $ the
standard normal series of ${\rm{Aut}}\mathscr{A}.$

$\bullet$ \textit{Homogeneous admissible automorphism groups}

 Let
$${\rm{Aut}}^{*}(\mathcal{O}:\mathfrak{g}):={\rm{Aut}}^{*}\mathcal{O}\cap
{\rm{Aut}}(\mathcal{O}:\mathfrak{g}).$$ We call
${\rm{Aut}}^{*}(\mathcal{O}:\mathfrak{g})$ a homogeneous admissible
automorphism group of $\mathcal{O}.$

$\bullet$ \textit{Standard normal series of
${\rm{Aut}}(\mathcal{O}:\mathfrak{g})$}

Put
$${\rm{Aut}}_{i}(\mathcal{O}:\mathfrak{g}):={\rm{Aut}}_{i}\mathcal{O}\cap
{\rm{Aut}}(\mathcal{O}:\mathfrak{g}),\,\ i\geq 0.$$ According to
Lemma \ref{zz1}, the standard filtration of $\mathcal{O}$ is
invariant under ${\rm{Aut}}(\mathcal{O}:\mathfrak{g}).$ It is
natural to define the standard normal series of
${\rm{Aut}}(\mathcal{O}:\mathfrak{g}),$ that is
$${\rm{Aut}}_{0}(\mathcal{O}:\mathfrak{g})> {\rm{Aut}}_{1}(\mathcal{O}:\mathfrak{g})> \cdots.$$

Our object is to review the connections between the automorphism
groups of $\mathfrak{g}$ and the automorphism groups of the
corresponding underlying superalgebra $\mathcal{O}.$ Let us define a
mapping,
\begin{eqnarray*}
\Phi:{\rm{Aut}}(\mathcal{O}:\mathfrak{g})&\longrightarrow&
{\rm{Aut}}\mathfrak{g}\\
\sigma&\longmapsto&\widetilde{\sigma}\mid_{\mathfrak{g}}
   \end{eqnarray*}
 where $\widetilde{\sigma}(D):=\sigma
D\sigma^{-1}$ for $D\in \mathfrak{g}.$ In the following, we are
going to study in detail the automorphism groups via the mapping. To
the aim, we require several lemmas.

\begin{lemma}\label{zz2} {\rm{(see \cite{lh})}} The following properties hold:
 \begin{eqnarray*}
 &{\rm{(1)}}&
  SHO'=\overline{SHO}\bigoplus{\rm{span}}_{\mathbb{F}}\{{\rm{T_{H}}}(x^{(\alpha)}x^{u})\mid
I(\alpha,u)=\widetilde{I}(\alpha,u)=\varnothing\};\\
 &{\rm{(2)}}&
\overline{SHO}=SHO\bigoplus
 {\rm{span}}_{\mathbb{F}} \big\{\big[{\rm{T_{H}}}(x^{(\pi)}),{\rm{T_{H}}}(x^{\omega})\big]\big\};\\
 &{\rm{(3)}}&
 \overline{SHO}=\bigoplus^{\xi-4}_{i=-1}\overline{SHO}_{[i]},\,\ SHO=\bigoplus^{\xi-5}_{i=-1}SHO_{[i]},\,\ \xi=\sum_{i\in
 Y_{0}}p^{t_{i}};\\
&{\rm{(4)}}&
  \overline{SHO}_{[i]}=SHO_{[i]},\,\
  \overline{SHO}_{[i]}=\big[\overline{SHO}_{[-1]},\overline{SHO}_{[i+1]}\big];\\
&{\rm{(5)}}&
  \overline{SHO}_{[\xi-4]}=SHO'_{[\xi-4]}=\mathbb{F}\big[{\rm{T_{H}}}(x^{(\pi)}),{\rm{T_{H}}}(x^{\omega})\big],
  \end{eqnarray*}
 in particular,
\rm{dim}$\overline{SHO}_{[\xi-4]}$= \rm{dim}$SHO'_{[\xi-4]}$= 1.
\end{lemma}

\begin{lemma}\label{zz1}
 The standard filtration of $X$ is invariant under automorphism of $X,$ where
 $X=\mathcal{O}, \mathfrak{g} .$
\end{lemma}
\begin{proof}
For $X=SHO'$ or $\mathcal{O},$ the conclusion follows from
\cite{lz,hl}. For $X=\overline{SHO}$ or $SHO ,$ the proof is the
same manner as that of \cite[Theorem 2]{hl}.
\end{proof}

\begin{lemma}\label{zz3}
Let $\big(G,[p]\big)$ be a restricted Lie superalgebra. Assume $L$
is a subalgebra of $G$ with a sum of superspaces $L=[L,L]+ V.$ If
$[L,L]$ is a restricted subalgebra of $G$ and
$V_{{\overline{0}}}^{[p]}$ is contained in $L_{\overline{0}},$ then
$(L,[p])$ is
 restricted.
\end{lemma}
\begin{proof}
It is enough to prove $x^{[p]}\in L_{\overline{0}},$ for arbitrary
$x\in L_{\overline{0}}.$ Since $L=[L,L]\bigoplus V,$ we can express
$x=y+z,$ where $y\in [L,L]_{\overline{0}}, z\in V_{\overline{0}}.$
Because $x^{[p]}=y^{[p]}+z^{[p]}+\sum_{i=1}^{p-1}S_{i}(y,z),$ where
$\sum_{i=1}^{p-1}S_{i}(y,z)\in [L,L]_{\overline{0}},$  $z^{[p]}\in
L_{\overline{0}},$ we can conclude that $x^{[p]}\in
L_{\overline{0}}\,\ \mbox{for all}\,\ x\in L_{\overline{0}}.$
\end{proof}
The following conclusion is analogous to those of other Cartan type
Lie superalgebras (see \cite{lz, lz2}).

\begin{proposition}\label{zz4}
$\mathfrak{g}(m,m;\underline{t})$ is restricted if and only if
 $\underline{t}=\underline{1}.$
\end{proposition}
\begin{proof}
Suppose $\mathfrak{g}(m,m;\underline{t})$ is restricted. Then
$({\rm{ad}}D_{i})^{p}$ are inner derivations of $\mathfrak{g}$ for
all $i\in \mathbf{I}_{0}.$ Hence ${\rm{zd}}({\rm{ad}}D_{i})^{p}\geq
-1.$ On the other hand, noticing that
${\rm{zd}}({\rm{ad}}D_{i})=-1,$ we obtain
${\rm{zd}}({\rm{ad}}D_{i})^{p}=-p.$ As a consequence,
$({\rm{ad}}D_{i})^{p}=0$ for all $i\in \mathbf{I}_{0}.$ Assert
$\underline{t}=\underline{1}.$ Otherwise
$\underline{t}>\underline{1},$ applying $({\rm{ad}}D_{i})^{p}$ to
${\rm{T_{H}}}(x^{(p+1)\varepsilon_{i}}) \in \mathfrak{g},$ we get
$({\rm{ad}}D_{i})^{p}\neq0,$ where $i\in \mathbf{I}_{0}.$ But this
violates $({\rm{ad}}D_{i})^{p}=0$ for all $i\in \mathbf{I}_{0}.$

 If $\underline{t}=\underline{1},$ then $W(m,m;\underline{1})$
is a restricted Lie superalgebra with respect to the usual
  $p$-mapping. (see \cite{Zh}).

  For $\mathfrak{g}=SHO,$ the conclusion follows directly from
  \cite{bln}. For $\mathfrak{g}=\overline{SHO},$ in view of Lemma \ref{zz2}(2),
for arbitrary $\overline{y}\in \overline{SHO}_{\overline{0}},$
$$\overline{y}=y +\lambda\big[{\rm{T_{H}}}(x^{(\pi)}),{\rm{T_{H}}}(x^{\omega})\big],$$
where $y\in SHO_{\overline{0}}, \lambda\in \mathbb{F},\,\
\mbox{and}\,\
\big[{\rm{T_{H}}}(x^{(\pi)}),{\rm{T_{H}}}(x^{\omega})\big]$ is even.
Thanks to Lemma \ref{zz3}, it remains only to observe that
$\big[{\rm{T_{H}}}(x^{(\pi)}),{\rm{T_{H}}}(x^{\omega})\big]^{p}$
does lie in $\overline{SHO}_{\overline{0}}.$ In fact, a direct
computation shows
$\big[{\rm{T_{H}}}(x^{(\pi)}),{\rm{T_{H}}}(x^{\omega})\big]^{p}=0.$

For $\mathfrak{g}=SHO',$ according to Lemma \ref{zz2}(1) and Lemma
\ref{zz3}, it remains only to consider that
${\rm{T_{H}}}(x^{(\alpha)}x^{u})^{p} \in SHO'_{\overline{0}}$ ,
where ${\rm{T_{H}}}(x^{(\alpha)}x^{u})$ is even and
$I(\alpha,u)=\widetilde{I}(\alpha,u)=\varnothing.$ Note that
\[{\rm{T_{H}}}(x^{(\alpha)}x^{u})^{p}=\left\{
    \begin{array}{cc}{\rm{T_{H}}}(x^{(\alpha)}x^{u})&\mbox{if $\alpha=\varepsilon_{i}$  and $u=<i'> $ for some $i\in
    Y_{0},$}\\0&\mbox{otherwise}.\end{array} \right.
    \]
\end{proof}

 For the rest of this section,  we shall restrict our attention to the
restrictedness case. Suppose $\mathfrak{g}$ is
   restricted, that is $\mathfrak{g}=\mathfrak{g}(m,m;\underline{1}),$ and
   correspondingly, $X:=X(m,m;\underline{1}),$ where $X=\mathcal{O}, W.$ In the following, let ${\rm{M}}_{2m}(\mathcal{O})$ be the
$\mathbb{F}$-algebra of all $2m\times2m$ matrices over
$\mathcal{O}.$ Denote by ${\rm{pr}}_{[0]}$ and ${\rm{pr}}_{1}$ the
projections of $\mathcal{O}$ onto $\mathcal{O}_{[0]}=\mathbb{F}$ and
$\mathcal{O}_{1},$ respectively. For $A=(a_{ij})\in
{\rm{M}}_{2m}(\mathcal{O}),$ put
${\rm{pr}}_{[0]}(A):=\big({\rm{pr}}_{[0]}(a_{ij})\big)$ and
${\rm{pr}}_{1}(A):=\big({\rm{pr}}_{1}(a_{ij})\big).$
\begin{lemma}\label{zz5}The following statements hold:

{\rm{(1) (see \cite{lz})}} Suppose that $\{E_{1},\ldots,E_{2m}\}$ is
an $\mathcal{O}$-basis of $W.$ Then
$$\big\{{\rm{pr}}_{[-1]}(E_{1}),\ldots,{\rm{pr}}_{[-1]}(E_{2m})\big\}$$ is an $\mathbb{F}$-basis of
$W_{[-1]},$ where ${\rm{pr}}_{[-1]}$ is the projection of $W$ onto
$W_{[-1].}$

{\rm{(2)}}
 Let $L=\bigoplus_{i\geq -1}L_{[i]}$ denote a
finite dimensional $\mathbb{Z}$-graded subalgebra of $W$ and
$L_{[-1]}=W_{[-1]}.$ Suppose $\phi\in {\rm{Aut}}L$ and $\phi$
preserves the standard filtration $($that is, $\phi(L_{j})\subset
L_{j},\,\ \mbox{for all}\,\ j\geq -1)$. If $\{G_{i}\mid i\in
\mathbf{I}\}\subset L$ is an $\mathcal{O}$-basis of $W,$ so is
$\{\phi(G_{i})\mid i \in \mathbf{I}\}.$

 {\rm{(3)}} Let  $L=\bigoplus_{i\geq -1}L_{[i]}$ denote a
finite dimensional transitive $\mathbb{Z}$-graded subalgebra of $W$
and $L_{[-1]}=W_{[-1]}.$ Suppose $\sigma,\,\ \tau \in{\rm{Aut}}L$
  and $\sigma$, $\tau$ preserve the standard filtration of $L$. If
  $\sigma\mid_{L_{[-1]}}=\tau\mid_{L_{[-1]}},$ then $\sigma=\tau.$
 \end{lemma}
 \begin{proof}
 {\rm(2)} By assumption, $\phi$ may induce an $\mathbb{F}$-isomorphism
 of the quotient space
$$\overline{\phi}: L/L_{0}\longrightarrow L/L_{0},
\,\ {\rm {dim}} L/L_{0}=2m.$$
  Denote by $\overline{G_{i}}$ the image $G_{i}$ under the canonical map
  $L \longrightarrow L/L_{0}.$ Then $\{\overline{G_{i}}\mid i\in \mathbf{I}\}$ is an
  $\mathbb{F}$-basis of $L/L_{0}.$ Suppose
  $$(\phi(G_{1}),\ldots,\phi(G_{2m}))^{{\rm{T}}}=A(D_{1},\ldots, D_{2m})^{{\rm{T}}},$$
  where $A\in {\rm{M}}_{2m}(\mathcal{O}).$ It follows that
  $$(\phi(G_{1}),\ldots,\phi(G_{2m}))^{{\rm{T}}}={\rm{pr}}_{[0]}(A)(D_{1},\ldots, D_{2m})^{{\rm{T}}}+{\rm{pr}}_{1}(A)(D_{1},\ldots,D_{2m})^{{\rm{T}}}.$$
Since $W_{[-1]}=L_{[-1]},$ we have
$$(\overline{\phi}(\overline{G_{1}},\ldots,\overline{G_{2m}}))^{{\rm{T}}}=(\overline{\phi(G_{1})},\ldots, \overline{\phi(G_{2m})})^{{\rm{T}}}={\rm{pr}}_{[0]}(A)(\overline{D_{1}},\ldots, \overline{D_{2m}})^{{\rm{T}}}.$$
This implies that ${\rm{pr}}_{[0]}(A)\in {\rm{GL}}(m,m).$ From
\cite[Lemma 2]{lz}, we know $A$ is invertible and therefore
$\{\phi(G_{i})\mid i\in \mathbf{I}\}$ is an $\mathcal{O}$-basis of
$W.$

 {\rm{(3)}} Use induction on $k,$ the case $k=-1$ being $\sigma|_{L_{[-1]}}=\tau|_{L_{[-1]}}.$ For all $E\in L_{[k]},D_{i}\in L_{[-1]},$  we have $$\big[\sigma(E), \sigma(D_{i})\big]=\big[\tau(E), \tau(D_{i})\big]\in L_{[k-1]},$$
that is , $\big[(\sigma-\tau)(E), \sigma(D_{i})\big]=0,$ for all
$i\in \mathbf{I}.$ Since $L$ is transitive and the standard
filtration of $L$ is invariant under automorphism of $L,$ we have
$(\sigma-\tau)(E)=0,$ namely, $\sigma=\tau.$

 \end{proof}

 \begin{lemma}\label{zz6}
 Suppose that $\phi \in {\rm{Aut}}\mathfrak{g}.$ Then there exist $y_{j}\in \mathcal{O}_{1}$
 and ${\rm{p}}(y_{j})=\mu(j)$ such that
 $\phi(D_{i})(y_{j})=\delta_{ij}$ for
 $i,j\in \mathbf{I}.$ Furthermore, the matrix
 $\big(\phi(D_{i})(y_{j})\big)_{i,j\in \mathbf{I}}$ is invertible.
\end{lemma}
\begin{proof}
 In view of Lemma \ref{zz4}(2), since $\{D_{1},\ldots,D_{2m}\}$ is an $\mathcal{O}$-basis of $W,$  so is $\{\phi(D_{1}),\ldots,\phi(D_{2m})\}.$  For ${\rm{T_{H}}}(x_{1}x_{j})\in \mathfrak{g},$ where $j\in
 \mathbf{I}
\backslash {\{1^{'}\}},$ we may assume that
$$\phi\big({\rm{T _{H}}}(x_{1}x_{j})\big
)=\sum^{2m}_{l=1}a_{jl}\phi(D_{l}),\,\ a_{jl}\in \mathcal{O}.$$ In
particular, $\phi$ preserves the filtration, and thanks to Lemma
\ref{zz4}, this forces $a_{jl}\in \mathcal{O}_{1}.$ We can obtain
that
\begin{eqnarray}
\phi\big([D_{i},{\rm{T_{H}}}(x_{1}x_{j})]\big)&=&\big[\phi(D_{i}),\sum_{l=1}^{2m}a_{jl}\phi(D_{l})\big]\\
&=&\sum_{l=1}^{2m}\big(\phi(D_{i})(a_{jl})\big)\phi(D_{l})\label{e1}.
 \end{eqnarray}
On the other hand,
 since $[D_{i},{\rm{T_{H}}}(f)]={\rm{T_{H}}}(D_{i}(f))$, whence
\begin{eqnarray}\label{e2}
&\phi\big([D_{i},{\rm{T_{H}}}(x_{1}x_{j})]\big)=\delta_{ij}\phi(D_{1^{'}})+(-1)^{\mu(j)}\delta_{i1}\phi(D_{j^{'}}).
\end{eqnarray}
Comparing (\ref{e1}) with (\ref{e2}), we have
$\phi(D_{i})(a_{j1'})=\delta_{ij}.$ Put $y_{j}:=a_{j1'}$ for all
$j\in \mathbf{I}\backslash \{1'\}.$ When $j\in \mathbf{I}\backslash
\{1'\},$ we know that $\phi(D_{i})(y_{j})=\delta_{ij},\,\ y_{j}\in
\mathcal{O}_{1}$ and
${\rm{p}}(y_{j})={\rm{p}}(a_{j1'})=\mu(j')+\mu(1')=\mu(j).$

Evidently, the element
 ${\rm{T_{H}}}(x_{1}x_{1'}-x_{2}x_{2'})$ lies in
$\mathfrak{g}.$ Let
$$\phi\big({\rm{T_{H}}}(x_{1}x_{1'}-x_{2}x_{2'})\big)=\sum^{2m}_{l=1}a_{l}\phi(D_{l}).$$
Then
\begin{eqnarray}
\phi\big([D_{i},{\rm{T_{H}}}(x_{1}x_{1'}-x_{2}x_{2'})]\big)&=&\sum_{l=1}^{2m}\big(\phi(D_{i})(a_{l})\big)\phi(D_{l}).
 \end{eqnarray}
On the other hand,
\begin{eqnarray}
\phi\big([D_{i},{\rm{T_{H}}}(x_{1}x_{1'}-x_{2}x_{2'})]\big)&=&\phi(D_{1'})-\phi(D_{1})-\phi(D_{2'})+\phi(D_{2}).
 \end{eqnarray}
Similar to the case above, we put $y_{j}=a_{1'}.$ Now
 ${\rm{p}}(a_{1'})=\mu(1')$
and $\phi(D_{i})(y_{j})=\delta_{i1'},$ as desired.
\end{proof}

Our main intent is to describe the connections between
$\rm{Aut}\mathfrak{g}$ and $\rm{Aut}(\mathcal{O}:\mathfrak{g}).$
Recall that
\begin{eqnarray*}
\Phi:{\rm{Aut}}(\mathcal{O}:\mathfrak{g})&\longrightarrow& {\rm{Aut}}\mathfrak{g}\\
\sigma&\longmapsto&\widetilde{\sigma}\mid_{\mathfrak{g}}
\end{eqnarray*}
where $\widetilde{\sigma}(D)=\sigma D\sigma^{-1},$ for all $D\in
\mathfrak{g}.$
\begin{theorem}\label{zz7}
$\Phi$ is an isomorphism of groups.
\end{theorem}
\begin{proof}

Clearly,
$$\Phi:{\rm{Aut}}(\mathcal{O}:\mathfrak{g})\longrightarrow
{\rm{Aut}}\mathfrak{g},\,\
 \sigma \longmapsto
\tilde{\sigma}\mid_{\mathfrak{g}}$$ is a homomorphism of groups,
where $\tilde{\sigma}(D)=\sigma D\sigma^{-1}£¬$ for all $D\in
\mathfrak{g}.$ Therefore, we shall merely prove $\Phi$ is bijective.

   Let us first show that $\Phi$ is injective, that is $\rm{ker}(\Phi)=\{1_{\mathcal{O}}\}.$
   To that aim, let
   $\sigma\in \rm{Aut}(\mathcal{O}:\mathfrak{g})$ such that
   $\tilde{\sigma}\mid_{\mathfrak{g}}=1_{\mathfrak{g}}.$  Let ${\rm{T_{H}}}(x_{j})\in
   \mathfrak{g},\,\ j\in \mathbf{I}.$
    Since $\widetilde{\sigma}\mid_{\mathfrak{g}}=1_{\mathfrak{g}}$ and the following
    equations hold:
   $${\rm{T_{H}}}(x_{j})(x_{k})=\sigma\big({\rm{T_{H}}}(x_{j})(x_{k})\big)={\rm{T_{H}}}(x_{j})\big(\sigma(x_{k})\big).$$
   We conclude that $x_{k}-\sigma(x_{k})\in \mathbb{F},$ $k\in \mathbf{I}.$
   On the other hand, Lemma \ref{zz1}(1) ensures that $\sigma(x_{k})\in
   \mathcal{O}_{1}.$ It follows that $\sigma(x_{k})-x_{k}\in \mathcal{O}_{1}\cap \mathbb{F},$ so $\sigma(x_{k})=x_{k}$ for all $k\in \mathbf{I}.$
   As $\mathcal{O}$ is generated by $x_{r},$ $r\in \mathbf{I},$ this implies
   $\sigma=1_{\mathcal{O}}.$

     The remaining task is to show that $\Phi$ is surjective. Let $\phi\in {\rm{Aut}}\mathfrak{g}.$ By the preceding Lemma \ref{zz5}, there exist $y_{j}\in
     \mathcal{O},$ with ${\rm{p}}(y_{j})=\mu({j}),$ such that $\phi(D_{i})(y_{j})=\delta_{ij}.$ Suppose $\phi(D_{i})=\sum^{2m}_{j=1}a_{ij}D_{j},\,\ a_{ij}\in \mathcal{O}.$
     Then the matrix $\big(\phi(D_{i})(y_{j})\big)$ is equal to
     $(a_{ij})(D_{i}y_{j})$ and therefore, $$
     (\delta_{ij})=\big(\phi(D_{i})(y_{j})\big)={\rm{pr}}_{[0]}\big(\phi(D_{i})(y_{j})\big)={\rm{pr}}_{[0]}(a_{ij}){\rm{pr}}_{[0]}(D_{i}y_{j}).$$
     It follows that ${\rm{pr}}_{[0]}(D_{i}y_{j})$ is invertible. Define the
     endomorphism of $\mathcal{O}$ such that
     \begin{eqnarray}\label{e3}
     &\sigma(x_{i})=y_{j},\,\ i,j \in \mathbf{I}.
     \end{eqnarray}
     Then $\sigma$ is even. We assert that $\sigma\in {\rm{Aut}}\mathcal{O}.$ In fact, from
     (\ref{e3}), it is easy to see that $\sigma$ preserves the standard
     filtration of $\mathcal{O}$ invariant, that is, $\sigma(\mathcal{O}_{i})\subset
     \mathcal{O}_{i}$ for all $i\geq 0.$ Therefore it induces a linear
     transformation $\sigma_{i}$ of $\mathcal{O}_{i}/\mathcal{O}_{i+1}.$
     Note that the matrix of $\sigma_{1}$ relative to $\mathbb{F}$-basis
     $\{x_{1}+\mathcal{O}_{2},\ldots,x_{2m}+\mathcal{O}_{2}\}$ is just
     $\big({\rm{pr}}_{[0]}(D_{i}y_{j})\big).$ This implies that $\sigma_{1}$ is
     bijective. Proceeding by induction on $i\geq 1,$ we get $\sigma_{i}$ is bijective. Hence $\sigma$ is bijective.  Now
     our assertion
     follows.
     Note that $$\widetilde{\sigma}(D_{i})(y_{j})=(\sigma
     D_{i}\sigma^{-1})(y_{j})=\sigma(D_{i}x_{j})=\delta_{ij}=\phi(D_{i})(y_{j}),$$
     for all $i,j\in \mathbf{I}.$
     Since $\{y_{j}\mid j\in \mathbf{I}\}$ generated $\mathcal{O},$ we conclude that
     $\widetilde{\sigma}(D_{i})=\phi(D_{i}), i\in \mathbf{I}.$ Because of Lemma \ref{zz4}(3), this makes it clear that
     $\widetilde{\sigma}\mid_{\mathfrak{g}}=\phi.$ The Theorem follows at once.
     \end{proof}

Note that in the next theorem we give a more thorough discussion
about the relations between ${\rm{Aut}}(\mathcal{O}:\mathfrak{g})$
and ${\rm{Aut}}\mathfrak{g}$. Namely, $\Phi$ preserves the standard
normal series and the homogeneous automorphism groups of
$\mathfrak{g}$.

 \begin{theorem}\label{zz8}
With the above notation $\Phi,$ the following identities hold:

${\rm{(1)}}\,\ \Phi({\rm{Aut}}_{i}(\mathcal{O}:\mathfrak{g}))=
{\rm{Aut}}_{i}\mathfrak{g} \,\ \mbox{for}\,\ i\geq 0;$

 ${\rm{(2)}}\,\ \Phi({\rm{Aut}}^{*}(\mathcal{O}:\mathfrak{g}))=\rm{Aut}^{*}\mathfrak{g}.$

\end{theorem}
\begin{proof}
$(1)$ \lq \lq $ \subset $\rq \rq   Let $\sigma\in
{\rm{Aut}}_{i}(\mathcal{O}:\mathfrak{g}).$ Then
     $\sigma^{- 1}\in {\rm{Aut}}_{i}(\mathcal{O}:\mathfrak{g}).$ For $k\in \mathbb{N}$ and $f\in
     \mathcal{O}_{k},$ we may assume that $\sigma^{-1}(f)=f+f', f'\in
     \mathcal{O}_{i+k},$ and $$\sigma(D_{j}(f))=D_{j}(f)+f'', f''\in
     \mathcal{O}_{i+k-1}.$$ According to Lemma \ref{zz1}, we have $\sigma(D_{j}(f'))\in
     \mathcal{O}_{i+k-1}.$ Note that
     $$\widetilde{\sigma}(D_{j})(f)=D_{j}(f)+f''+\sigma(D_{j}(f')).$$
     We obtain that $\widetilde{\sigma}(D_{j})(f)\equiv
     D_{j}(f)\pmod{\mathcal{O}_{i+k-1}}.$ This implies that
     $$\widetilde{\sigma}(D_{j})\equiv D_{j}\pmod{ W_{i-1}}\quad \mbox{for
     all}\,\
      j\in \mathbf{I}.$$
     A straightforward calculation shows
     $$ \widetilde{\sigma}(fD_{j})=\sigma(f)\widetilde{\sigma}(D_{j})\quad  \mbox{for all}
     \,\ j\in \mathbf{I}.$$
     Then it is easy to see that $$\widetilde{\sigma}(fD_{j})\equiv
     fD_{j}\pmod{W_{k-1+i}}.$$
     Therefore $\widetilde{\sigma}\in {\rm{Aut}}_{i}W\cap {\rm{Aut}}\mathfrak{g}\subset
     {\rm{Aut}}_{i}\mathfrak{g}.$
     Hence $\Phi({\rm{Aut}}_{i}(\mathcal{O}:\mathfrak{g}))\subset {\rm{Aut}}_{i}\mathfrak{g}.$

      \lq \lq $\supset$\rq \rq   Let $\widetilde{\sigma}\in {\rm{Aut}}_{i}\mathfrak{g}, i\geq 0$ and set
         $\sigma:=\Phi^{-1}(\widetilde{\sigma}).$
         Given $j\in \mathbf{I},$ pick $k\in \mathbf{I}\setminus
         \{j'\}.$ We have
         \begin{eqnarray}
         {\rm{T_{H}}}(x_{k'}x_{j})=(-1)^{\mu(k')+\mu(k')\mu(j)}x_{j}D_{k}+(-1)^{\mu(j)}x_{k'}D_{j'}.\label{e4}
         \end{eqnarray}
         Then
         \begin{eqnarray}
         (-1)^{\mu(k')+\mu(k')\mu(j)}\sigma(x_{j})(\widetilde{\sigma}(D_{k}))+(-1)^{\mu(j)}\sigma(x_{k'})(\widetilde{\sigma}(D_{j'}))
         =\widetilde{\sigma}({\rm{T_{H}}}(x_{k'}x_{j}))\\
         \equiv(-1)^{\mu(k')+\mu(k')\mu(j)}x_{j}D_{k}+(-1)^{\mu(j)}x_{k'}D_{j^{'}}\pmod{
         \mathfrak{g}_{i}}.\label{e5}
         \end{eqnarray}
         Notice that $\widetilde{\sigma} \in {\rm{Aut}}_{i}\mathfrak{g}$ and $W_{[-1]}=\mathfrak{g}_{[-1]}.$
         We have
         \begin{eqnarray}\widetilde{\sigma}(D_{k})=D_{k}+E_{1},\,\ \widetilde{\sigma}(D_{j'})=D_{j'}+E_{2},\mbox{where}\,\ E_{1},E_{2}\in
         \mathfrak{g}_{i-1}.\end{eqnarray}
         By Lemma \ref{zz1}, it is easily seen that
         $\sigma(x_{j})E_{1}, \sigma(x_{k'})E_{2} \in W_{i}.$
         Thus from (\ref{e4}) and (\ref{e5}), we obtain that
         $$(-1)^{\mu(k')+\mu(k')\mu(j)}(\sigma(x_{j})-x_{j})D_{k}+(-1)^{\mu(j)}(\sigma(x_{k'})-x_{k'})D_{j'}\equiv
         0\pmod{W_{i}}.$$
         Since $k'\neq j,$ we obtain that $\sigma(x_{j})\equiv
         x_{j}\pmod {\mathcal{O}_{i+1}}.$
         Now using induction on $|\alpha|+|u|,$ one
         may prove that $$\sigma(x^{(\alpha)}x^{u})\equiv
         x^{(\alpha)}x^{u}\pmod{\mathcal{O}_{\mid\alpha\mid+\mid u\mid+i}},$$
         This implies $\sigma\in {\rm{Aut}}_{i}\mathcal{O}$ and therefore $\sigma\in
         {\rm{Aut}}_{i}(\mathcal{O}:\mathfrak{g}).$
         Hence $$\Phi({\rm{Aut}}_{i}(\mathcal{O}:\mathfrak{g}))\supset {\rm{Aut}}_{i}\mathfrak{g}.$$\

         (2) An analogous argument with the case above leads to the
         conclusion.
\end{proof}
 Next theorem is a general conclusion for a finite dimensionial
 $\mathbb{Z}$-graded superalgebra $\mathscr{A}$ which satisfies that ${\rm{Aut}}\mathscr{A}$
preserves the standard filtration invariant.

\begin{theorem}\label{zz9}
Suppose $\mathscr{A}=\bigoplus_{i\in{\mathbb{Z}}}\mathscr{A}_{[i]}$
is a finite dimensional $\mathbb{Z}$-graded superalgebra over
$\mathbb{F}$ and the standard filtration of $\mathscr{A}$ is
invariant under ${\rm{Aut}}\mathscr{A}.$ Then
${\rm{Aut}}_{i}\mathscr{A}\,\ (i\geq 1)$ are solvable normal
subgroups of ${\rm{Aut}}\mathscr{A}.$
 \end{theorem}
 \begin{proof}
Since ${\rm{Aut}}\mathscr{A}$ preserves the standard filtration
invariant, we have
$$[{\rm{Aut}}_{i}\mathscr{A},{\rm{Aut}}_{j}\mathscr{A}]\subset
{\rm{Aut}}_{i+j}\mathscr{A},\,\ i,j\geq0.$$ From this we deduce that
the normal series $${\rm{Aut}}_{1}\mathscr{A} >
{\rm{Aut}}_{2}\mathscr{A} > \cdots $$ is abelian (that is,
${\rm{Aut}}_{i}\mathscr{A}/{\rm{Aut}}_{i+1}\mathscr{A}$ are abelian
groups, for all $i\geq 1$) and reaches 0. This shows that
${\rm{Aut}}_{i}\mathscr{A}$ is solvable.
\end{proof}


\begin{thebibliography}{99}
\bibitem{bln} W. Bai, W.-D. Liu and L. Ni. Derivations of the
finite-dimensional special odd Hamiltonian Lie superalgebras.
\textit{arXiv:} 1007. 1098 \textit{math. RT.}
\bibitem{bl} S. Bouarroudj and D. Leites. Simple Lie superalgebras
and nonintegrable distributions in characteristic $p$. \textit{J.
Math. Sci.} \textbf{141} (4) (2007): 1390-1398.
\bibitem{FJZ} J.-Y. Fu, Q.-C. Zhang and C.-P. Jiang. The Cartan-type
modular Lie superalgebra $KO$. \textit{Commun. Algebra.} \textbf{34}
(1) (2006): 107-128.
\bibitem{hl} Y.-H. He, W.-D. Liu and B. Li. Filtration structure of finite
dimensional special odd Hamiltonian superalgebras in prime
characteristic. \textit{J. Beijing Inst.
Technol.} \textbf{18} (4) (2009): 488-491.
\bibitem{lh} W.-D. Liu and Y.-H. He. Finite dimensional
special odd Hamiltonian superalgebras in prime characteristic.
\textit{Commun. Contemp. Math.} \textbf{11} (4) (2009): 523-546.
\bibitem{lz} W.-D. Liu and Y.-Z. Zhang. Automorphism groups of restricted
Cartan-type Lie superalgebras. \textit{Commun. Algebra.} \textbf{34}
(2006): 3767-3784.
\bibitem{lz2} W.-D. Liu and Y.-Z. Zhang. Finite-dimensional odd
Hamiltonian superalgebras over a field of prime characteristic.
\textit{J. Aust. Math. Soc.} \textbf{79} (2005): 113-130.
\bibitem{RLW} R. L. Wilson. Automorphisms of graded Lie algebras of
Cartan type. \textit{Commun. Algebra.} \textbf{3} (7) (1975):
591-613.
\bibitem{Sch} M. Scheunert. Theory of Lie superalgebras. \textit{Lecture Notes Math.}
\textbf{716} (1979), Springer-Verlag.
\bibitem{Zh} Y.-Z. Zhang. Finite-dimensional Lie superalgebras of
Cartan type over fields of prime characteristic. \textit{Chin. Sci.
Bull.} \textbf{42} (1997): 720-724.
\end{thebibliography}
\end{document}